\documentclass[11pt,twoside,reqno]{amsart}
\usepackage[all]{xy}
\usepackage{amssymb,latexsym,amsthm,amsmath,mathtools}
\usepackage{enumitem,color}

\usepackage[colorlinks=true, linkcolor=black, citecolor=blue]{hyperref}
\usepackage[capitalize,nameinlink]{cleveref}
\usepackage{cleveref}
\usepackage[
    backend=biber,
    style=numeric-comp,
    doi=true,
    url=false,
    maxbibnames=99,
    giveninits=true
]{biblatex}
\crefname{section}{Section}{Sections}
\crefname{subsection}{Subsection}{Subsections}
\crefname{subsubsection}{Subsubsection}{Subsubsections}
\usepackage{parskip}
\usepackage{mathtools}

\usepackage[backend=biber, doi=true, url=false, maxbibnames=99, giveninits=true]{biblatex}
\usepackage{orcidlink}
\long\def\symbolfootnote[#1]#2{\begingroup%
\usepackage{hyperref}
\usepackage{tikz}
\usepackage{blkarray}
\usepackage{arydshln}
\usepackage{tabularx}
\setlength\dashlinedash{0.5pt}
\setlength\dashlinegap{1.5pt}
\setlength\arrayrulewidth{0.3pt}
\def\thefootnote{\fnsymbol{footnote}}\footnote[#1]{#2}\endgroup}

\makeatletter
\def\imod#1{\allowbreak\mkern10mu({\operator@font mod}\,\,#1)}
\makeatother

\newtheorem{theorem}{Theorem}[section]
\newtheorem{lemma}[theorem]{Lemma}
\newtheorem{corollary}[theorem]{Corollary}
\newtheorem{proposition}[theorem]{Proposition}
\newtheorem*{theorem*}{Theorem}
\theoremstyle{definition}

\newtheorem{remark}[theorem]{Remark}

\newtheorem{question}[theorem]{Question}

\numberwithin{equation}{section}

\newcommand{\ignore}[1]{}

\newcommand{\mynote}[1]{}
\begin{document} 
\author[Saini P.]{Prachi Saini\,\orcidlink{0009-0006-7627-1538}}
\address{IISER Pune, Dr. Homi Bhabha Road, Pashan, Pune 411 008, India}
\email{prachi2608saini@gmail.com}
\author[Singh A.]{Anupam Singh}
\address{IISER Pune, Dr. Homi Bhabha Road, Pashan, Pune 411 008, India}
\email{anupamk18@gmail.com}
\thanks{The first-named author acknowledges the support of CSIR PhD scholarship number 09/0936(1237)/2021-EMR-I. The second-named author is funded by an ANRF-MATRICS Grant ANRF/ARGM/2025/000095/MTR}

\setcounter{section}{0}
\title{Polynomial Maps with Constants on Matrix Algebra}

\begin{abstract}
Let $\mathcal A$ be an $\mathbb F$-algebra and $\omega \in \mathcal A\langle x_1, \ldots, x_m \rangle$ which defines a map $\mathcal A^m \rightarrow \mathcal A$ by evaluation, called a polynomial map with constant. We consider $\mathcal {A} = M_n(\mathbb{F})$, the algebra of $n \times n$ matrices over an algebraically closed field $\mathbb{F}$ of characteristic $0$, and polynomial maps given by $\omega(x_1, x_2) = A_1x_1^k + A_2x_2^k$, where $A_1,A_2\in M_n(\mathbb F)$. For $n=2$, the images of such a map is competely determined in an earlier work (Panja, S.; Saini, P.; Singh, A., Images of polynomial maps with constants, Mathematika 71 (2025), no. 3, Paper No. e70031). In this article, by assuming one of the coefficients, say $A_1$, is invertible, we relate the surjectivity of $\omega$ to the nullity of $A_2$. When $n=3, 4$, we completely classify the surjectivity of $\omega(x_1, x_2)$ by obtaining the necessary and sufficient condition in terms of $n$, $k$, and the nullity of $A_2$.
\end{abstract}
\maketitle

\section{Introduction}
The L\'{v}ov-Kaplansky conjecture asks for the image of multilinear polynomials on a matrix algebra over an infinite field. This is known so far for degree $2$ algebras, see \cite{MR2846315, MR3195161} and for several special cases of $\mathbb F$ for degree 3 matrix algebra \cite{MR3415572}, and is far from being solved. The more general problems are to find out images of a polynomial map on a central simple algebra, including matrix Waring problems \cite{MR1286824, MR797536, MR4255750}. We propose to consider a yet broader class of problems: studying polynomial maps with constants on algebras. The study of images of noncommutative polynomials on matrix algebras has attracted considerable attention; see, e.g., \cite{MR2846315, 1936361}. In this context, Waring-type questions ask whether every element of an algebra can be expressed as a sum of a bounded number of values of a fixed polynomial map. Substantial progress has been made in recent years in understanding such problems across various algebraic settings, including upper triangular matrix algebras and Lie algebras; see, for instance, \cite{MR4623691, MR4453886, MR4607594, MR4526834, MR3670351,
MR4732194, MR4647477, MR4592319,MR4776363}.

Let $\mathbb F$ be a field and $\mathcal A$ be an $\mathbb F$-algebra. Consider $\mathcal F_m = \mathbb F\langle x_1, x_2, \dots, x_m \rangle$, the free algebra of rank $m$ and the free product $\mathcal A_m = \mathcal A * \mathcal F_m$. An element $\omega(x_1,\ldots,x_m)\in \mathcal A_m$ is a finite $\mathbb F$-linear combination of monomials of the form $a_0 x_{i_1}^{k_1} a_1 \cdots x_{i_r}^{k_r} a_r$,
where $a_j \in \mathcal A$ and $k_s \ge 1$. Such an element defines an evaluation map
\begin{align*}
\omega: \mathcal A^m &\longrightarrow \mathcal A,\\
(a_1,\dots,a_m) &\longmapsto \omega(a_1,\dots,a_m).
\end{align*}
These maps are called \emph{polynomial maps with constants}. Polynomial maps with constants arise naturally as analogs of word maps with constants in group theory. Word maps on groups have been intensively studied in recent years, particularly with regard to their images, distribution, and asymptotic behavior; see, for instance, \cite{MR4223709, MR4030419, MR3700228, MR3765462, MR3087161,   MR4694588}. A central direction in this area concerns Waring-type problems. This asks whether every element of a group can be expressed as a bounded product of images of a fixed word map, and several fundamental results have been proved. A fundamental result of Larsen, Shalev, and Tiep \cite{MR2846493} shows that for finite non-abelian simple groups of sufficiently large order, every element can be written as a product of two values of a fixed power word. This demonstrates that surjectivity can often be achieved using only two variables. 

Analogous phenomena arise in associative algebras. For instance, Kishore and Singh \cite{MR4930689} proved that over sufficiently large finite fields, every matrix is a sum of two $k$-th powers. Moreover, recent work \cite{MR4952110} shows that over an algebraically closed field and sufficiently large finite fields, diagonal polynomial maps with scalar coefficients in two variables are surjective on $M_n(\mathbb F)$. The same holds over $\mathbb R$, with some exceptions, in which negative scalars are missed. 

Motivated by these developments, we study polynomial maps with constants that need not be scalars. Let $k$ be a positive integer and $A_1, A_2\in \mathcal A$ be non-zero elements. Consider
$$
\omega(x_1, x_2) = A_1 x_1^{k} + A_2 x_2^{k} \in \mathcal A * \mathcal F_2.
$$
The main question is to understand the image of this map on $\mathcal A$.
\begin{question}
For which pairs $(A_1, A_2)\in \mathcal A^2$ is the map $\omega$ surjective on $\mathcal A$?
\end{question}

This problem has been studied in various settings. For the matrix algebra $M_2(\mathbb F)$, a necessary and sufficient condition for surjectivity was obtained in an earlier work \cite{MR4922424}, formulated in terms of simultaneous conjugation of the coefficients $A_1$ and $A_2$. The split octonion algebra $\mathcal O$ was considered in \cite{MR4996673}, where all those pairs that do not give surjectivity were explicitly classified. A common feature in these works is that surjectivity holds whenever at least one of the coefficients is invertible.

In this article, we investigate this problem for matrix algebra in the case that at least one of the coefficients is invertible. We relate the surjectivity of $\omega$ to the nullity of $A_2$ which will turn out to be the nilpotent blocks in the Jordan canonical form of the matrix $A_2$. Our approach shows that the nilpotent structure of this matrix plays a decisive role in determining the image of $\omega$. 

Our main results are the following.
\begin{theorem}
\label{thm:General criteria for surjectivity}
Let $\mathbb F$ be an algebraically closed field of characteristic $0$, and $k \geq 2$ an integer. For $A_1 \in \mathrm{GL}_n(\mathbb F)$ and $A_2 \in M_n(\mathbb F)$, consider the polynomial map $\omega : M_n(\mathbb F) \times M_n(\mathbb F) \rightarrow M_n(\mathbb F)$ given by
$$\omega(x_1, x_2) = A_1 x_1^k + A_2 x_2^k.$$ 
Let $r_0$ be the nullity of $A_2$. Then we have the following: 
\begin{enumerate}
\item If $r_0 \leq 1$, then $\omega$ is surjective on $M_n(\mathbb F)$.
\item If $n \leq k(r_0 - 1)$, then $\omega$ is not surjective on $M_n(\mathbb F)$.
\end{enumerate}
\end{theorem}
The proof of \cref{thm:General criteria for surjectivity} is given in \ref{proof: thm 1.2}. We use this for the case $n=3$ and $4$ and get the complete characterization of surjectivity.

\begin{theorem}
\label{thm: Low dimensional classification}
Let $\mathbb F$ be an algebraically closed field of characteristic $0$, and $k \geq 2$ an integer. Consider the polynomial map $\omega : M_n(\mathbb F) \times M_n(\mathbb F) \rightarrow M_n(\mathbb F)$ given by
$$\omega(x_1, x_2) = A_1 x_1^k + A_2 x_2^k$$
for $A_1 \in \mathrm{GL}_n(\mathbb F)$ and $A_2 \in M_n(\mathbb F)$,
Then, for $n=3$ and $4$, the map $\omega$ is surjective on $M_n(\mathbb F)$ if and only if $n > k(r_0 - 1)$, where $r_0$ is the nullity of $A_2$.
\end{theorem}
The proof of \cref{thm: Low dimensional classification} is given in \ref{proof: thm 1.3}. As a consequence, we obtain the following corollary.
\begin{corollary}
\label{coro: map A}
Let $\mathbb F$ be an algebraically closed field of characteristic $0$, and $k \geq 2$ an integer. For $A \in M_n(\mathbb F)$ non-zero, consider the polynomial map $\omega : M_n(\mathbb F) \times M_n(\mathbb F) \rightarrow M_n(\mathbb F)$ given by
$$\omega(x_1, x_2) = x_1^k + A x_2^k.$$ 
Then the surjectivity of $\omega$ on $M_n(\mathbb F)$, in terms of the nullity $r_0$ of $A$, is as follows:
\begin{enumerate}
    \item For $n=3$, the surjectivity of the map $\omega$ is described in \cref{tab:n3}. Moreover, when the map is not surjective, the image of $\omega$ is precisely the set 
    $$Z=M_3(\mathbb F)\setminus \left\{\begin{pmatrix}
        \mu & u\\
        0 & M\\
        \end{pmatrix}:\; \mu \in \mathbb F,\; u \in \mathbb F^2,\; M \in M_2(\mathbb F)\setminus\{0\},\; M^2 = 0\right\}.$$
    \item For $n=4$, the surjectivity of the map $\omega$ is described in \cref{tab:n4}.
\end{enumerate}
\end{corollary}
\begin{table}[htbp]
\centering
\renewcommand{\arraystretch}{1.3} 
\begin{tabular}{|c|c|c|} 
\hline nullity of $A$ & $k=2$ & $k \geq 3$ \\ 
\hline $r_0 = 0$ & Surjective & Surjective \\ 
\hline $r_0 = 1$ & Surjective & Surjective \\ 
\hline $r_0=2$ & Surjective & Not surjective (The image is exactly the set $Z$)\\ 
\hline \end{tabular} \vspace{0.2cm} 
\caption{Surjectivity of $\omega(x_1,x_2)=x_1^k + A x_2^k$ on $M_3(\mathbb F)$}
\label{tab:n3}
\end{table}

\begin{table}[htbp] 
\centering
\renewcommand{\arraystretch}{1.3} 
\begin{tabular}{|c|c|c|c|} 
\hline nullity of $A$ & $k=2$ & $k=3$ & $k \geq 4$ \\ 
\hline $r_0=0$ & Surjective & Surjective & Surjective \\ 
\hline $r_0=1$ & Surjective & Surjective & Surjective \\ 
\hline $r_0=2$ & Surjective & Surjective & Not Surjective \\ 
\hline $r_0=3$ & Not surjective & Not surjective & Not surjective \\ \hline \end{tabular} 
\vspace{0.2cm} 
\caption{Surjectivity of $\omega(x_1,x_2)=x_1^k + A x_2^k$ on $M_4(\mathbb F)$}
\label{tab:n4}
\end{table}

The structure of this article is as follows. In \cref{sec: Preliminaries}, we introduce the problem and reduce it to a more tractable form. In \cref{sec:atmost one nil block}, we investigate the surjectivity of the map in the case where the $\operatorname{nullity}(A_2)$ is at most one. We then turn, in \cref{sec: more than one nil block}, to the setting where $\operatorname{nullity}(A_2)\geq 2$. We establish a necessary condition for surjectivity in \cref{subsec: necessary condition}. The proof of \cref{thm:General criteria for surjectivity} is completed in \cref{proof: thm 1.2}. Finally, the low-dimensional cases are treated in \cref{proof: thm 1.3}, where we prove \cref{thm: Low dimensional classification}. The proof of Corollary~\ref{coro: map A} is given in Subsection~\ref{proof: coro 1.4}.

\subsection{Notation}
Throughout, $\mathbb F$ denotes an algebraically closed field of characteristic $0$, and the split central simple algebra over $\mathbb F$ is $M_n(\mathbb F)$, the algebra of all $n \times n$ matrices. We write the automorphism group of $M_n(\mathbb F)$ as $\mathrm{PGL}_n(\mathbb F)$ for the projective general linear group, which acts on $M_n(\mathbb F)$ by conjugation.

For a matrix $C \in M_n(\mathbb F)$, we denote its rows by $c_1, \dots, c_n \in \mathbb F^n$ and $C=(c_1, \ldots, c_n)^\top$. For $1 \leq i \leq n$, let $e_i$ denote the $i$-th standard basis vector of $\mathbb F^n$. For $C_1 \in M_{n_1}(\mathbb F)$ and $C_2 \in M_{n_2}(\mathbb F)$, we define their direct sum by $ C_1 \oplus C_2 = \begin{pmatrix}
C_1 & 0 \\ 0 & C_2 \end{pmatrix} \in M_{n_1+n_2}(\mathbb F)$. For $B \in M_n(\mathbb F)$, we write $B \sim \bigoplus_{i=1}^r J_{\lambda_i,n_i}, \qquad \sum_{i=1}^r n_i = n$, to denote a Jordan canonical form of $B$, where $J_{\lambda_i,n_i}$ is the Jordan block of size $n_i$ corresponding to the eigenvalue $\lambda_i$. We denote by $r_0(B)$ the nullity of $B$, which turns out to be the number of Jordan blocks corresponding to the eigenvalue $0$, and by $r'(B)$ the number of Jordan blocks corresponding to nonzero eigenvalues. Thus, $ r(B) = r_0(B) + r'(B)$. We define $n_0(B) = \sum_{\lambda_i = 0} n_i$, the dimension of the generalized eigenspace of $B$ corresponding to the eigenvalue $0$. If $J$ is a Jordan canonical form of $B$, we write $J'$ for the direct sum of all Jordan blocks corresponding to nonzero eigenvalues. The characteristic polynomial of $B$ is denoted by $\chi_B(z) \in \mathbb F[z]$.

Finally, let $T = (t_0,t_1,\dots,t_{n-1})^\top \in M_n(\mathbb F)$, and let $\mathcal{I} \subseteq \{1,2,\dots,n-1\}$. We define $T_{\mathcal{I}} \in M_{|\mathcal{I}| \times |\mathcal{I}|}(\mathbb F)$ to be the submatrix whose rows are indexed by $\mathcal{I}$ and whose columns are indexed by $\{i+1 : i \in \mathcal{I}\}$, that is $T_{\mathcal{I}} := (t_{i,\,j+1})_{i,j \in \mathcal{I}}$.

\section{Preliminaries}
\label{sec: Preliminaries}
In this section, we reduce the problem to a canonical form. 
Let $A_1, A_2 \in M_n(\mathbb{F})$ where $A_1$ is invertible, and let $k\geq 2$ be an integer. Consider the map
$$\omega : M_n(\mathbb{F})^2 \to M_n(\mathbb{F}), \qquad
\omega(x_1,x_2) = A_1x_1^k + A_2x_2^k.$$
We need to study the surjectivity of $\omega$.

\textbf{Reduction via conjugation.}
The property that $\omega$ is surjective is invariant under conjugation. More precisely, for any $g \in \mathrm{GL}_n(\mathbb{F})$, define
$$\omega_g(x_1,x_2) = (gA_1g^{-1}) x_1^{k} + (gA_2g^{-1}) x_2^{k}.$$
Then $\omega$ is surjective if and only if $\omega_g$ is surjective. Hence, it suffices to study the problem for pairs $(A_1, A_2)$ up to simultaneous conjugation.

\textbf{Reduction to a single matrix.}
Since we have $A_1$ invertible, define $B = A_1^{-1}A_2 \in M_n(\mathbb{F})$. Then $\omega$ is surjective if and only if the map
$$\widetilde{\omega}(x_1,x_2) = x_1^{k} + B x_2^{k}$$
is surjective. Indeed, $\omega(X_1,X_2)=C$ if and only if $\widetilde{\omega}(X_1,X_2)=A_1^{-1}C$.

\textbf{Reduction to Jordan form.}
Since $\mathbb{F}$ is algebraically closed, every matrix is conjugate to its Jordan canonical form. By the previous reduction, we may therefore assume that $B$ is conjugate to 
$$
J = \bigoplus_{i=1}^{r} J_{\lambda_i,n_i}, \qquad \sum_{i=1}^r n_i = n,
$$
where $J_{\lambda_i,n_i}$ denotes the Jordan block of size $n_i$ corresponding to the eigenvalue $\lambda_i \in \mathbb{F}$. 

We relate the number of nilpotent blocks of $B$ to the $\operatorname{nullity}(B)$.
\begin{lemma}
\label{lem:nullity-zero-blocks}
Let $B \in M_n(\mathbb F)$, and let $J$ be its Jordan canonical form over an algebraic closure of $\mathbb F$. Then $\dim \ker(B)$
is equal to the number of Jordan blocks of $J$ corresponding to the eigenvalue $0$.
\end{lemma}
\begin{proof}
Since similarity preserves the dimension of the kernel, we may assume that $B = J$. Write $J$ as a direct sum of Jordan blocks:
$$
J = \bigoplus_{i=1}^r J_{\lambda_i}.
$$
If $\lambda_i \neq 0$, then $J_{\lambda_i}$ is invertible, and hence $\ker(J_{\lambda_i}) = \{0\}$. If $\lambda_i = 0$, then $J_{\lambda_i}$ is a nilpotent Jordan block, and a direct computation shows that $\dim \ker(J_{\lambda_i}) = 1$.
Since the kernel of a block diagonal matrix is the direct sum of the kernels of its blocks, we obtain
$$
\dim \ker(J) = \#\{\text{Jordan blocks corresponding to } 0\}.
$$
The result follows.
\end{proof}
We denote by $r_0$ both the nullity of $B$ and the number of Jordan blocks of $B$ corresponding to the eigenvalue $0$.
\section{Surjectivity with at most one Jordan nilpotent block}\label{sec:atmost one nil block}

In this section, we consider matrices whose Jordan canonical form contains at most one block corresponding to the eigenvalue $0$.

\subsection{Both the coefficients are invertible}
We begin with the case in which both coefficients are invertible.
\begin{proposition}\label{thm:invertibleB}
Let $k\geq 2$ be a positive integer and let $B \in \mathrm{GL}_n(\mathbb F)$. Then the map $\widetilde{\omega}(x_1,x_2) = x_1^{k} + B x_2^{k}$ is surjective on $M_n(\mathbb F)$.
\end{proposition}
\begin{proof}
Let $C \in M_n(\mathbb F)$. Since $\mathbb F$ is algebraically closed, there exists $\lambda \in \mathbb F^\times$ such that $\lambda$ is not an eigenvalue of $C$. Set
$$
X_1 = \alpha I_n, \quad \text{where } \alpha^{k} = \lambda.
$$
Then $X_1^{k} = \lambda I_n$, and hence $C - X_1^{k}$ is invertible. It follows that $B^{-1}(C - X_1^{k}) \in \mathrm{GL}_n(\mathbb F)$. Over an algebraically closed field, every invertible matrix admits a $k$-th root. Thus there exists $X_2 \in M_n(\mathbb F)$ such that
$X_2^{k} = B^{-1}(C - X_1^{k})$. Therefore,
$$
C = X_1^{k} + B X_2^{k},
$$
which proves surjectivity.
\end{proof}

In view of Proposition~\ref{thm:invertibleB}, the map $\widetilde{\omega}$ is surjective whenever $B$ is invertible. Hence, it suffices to consider the case where $B$ is singular. In this setting, the structure of the nilpotent part of $B$, as encoded in its Jordan canonical form, plays a central role in the analysis.

\subsection{Full size nilpotent case}
We study the surjectivity of the map  $\widetilde\omega(x,y) = x^{k} + Jy^{k}$ where $J=J_{0,n}$, the (nilpotent) Jordan matrix of size $n$.
\begin{lemma}\label{lem:T-kth}
Let $T = (t_0, t_1, \dots, t_{n-1})^\top \in M_n(\mathbb F)$, where $t_0 = (1, 0, \ldots, 0)$. Suppose that exactly $m$ of the rows out of $t_1, \dots, t_{n-1}$ are zero, and let $ \mathcal I \subseteq \{1, \dots, n-1\}$ denote the set of indices of the nonzero rows. If $T_\mathcal I$ is invertible, then $T$ is a $k$-th power of a matrix in $M_n(\mathbb F)$ for every integer $k \geq 1$.
\end{lemma}
\begin{proof}
The characteristic polynomial of $T$ is  
\begin{equation*}
\chi_T(z)=\det(zI-T)=z^{m}(z-1) \chi_{T_\mathcal I}(z).
\end{equation*}
The eigenvalues of $T$ are, say, $\{0, 1, \lambda_1, \lambda_2,\dots,\lambda_{n-m-1}\}$ if $n-m-1 \geq 1$ with $\lambda_i \neq 0$. Otherwise, the eigenvalues are simply $0$ and $1$. In both cases, $0$ has algebraic multiplicity $m$. The matrix $T$ is similar to $\begin{pmatrix} A_{n-m} & 0\\ 0 & 0_{m} \end{pmatrix}$ where $A_{n-m}$ is a matrix whose eigenvalues are all the non-zero eigenvalues of $T$. Therefore, $A_{n-m}$ is invertible and can be written as the $k$-th power of some matrix in $M_{n-m}(\mathbb F)$. Since $0_{m}^k=0$, it follows that $T$ is a $k$-th power of a matrix in $M_n(\mathbb F)$.
\end{proof}
\begin{lemma}\label{lem:C-LI}
Let $C \in M_n(\mathbb F)$ have exactly $m < n$ linearly independent rows indexed by $\{i_1, i_2, \ldots, i_m\} \subset \{1, 2, \dots, n\}$ with $i_m = n$, and suppose all other rows are zero. Then there exist vectors $t_j \in \mathbb F^n$, for $j \in \{1, 2, \ldots, n\}\setminus \{i_1, i_2, \dots, i_m\}$, such that:
\begin{enumerate}
\item the set $\{c_{i_1}, c_{i_2}, \dots, c_{i_m}\} \cup \{t_j : j \notin \{i_1,\dots,i_m\}\}$ is linearly independent in $\mathbb F^n$, and
\item $\det(T_\mathcal I) \neq 0$, where $\mathcal I = \{j_1, j_2, \dots, j_{n-m}\} = \{1, 2,\dots, n\} \setminus \{i_1, i_2, \dots, i_m\}$.
\end{enumerate}
\end{lemma}
\begin{proof} Let $W = \mathrm{span} \{c_{i_1}, c_{i_2},\ldots, c_{i_m}\} \subset \mathbb F^n, \quad \dim W = m$. 
We construct the vectors $t_{j_1}, \dots, t_{j_{n-m}}$ inductively so that:
\label{method}
\begin{enumerate}
\item $\{c_{i_1}, \dots, c_{i_m}, t_{j_1},\dots,t_{j_{n-m}}\}$ is a basis of $\mathbb F^n$, and
\item the vectors $\pi(t_{j_1}),\dots,\pi(t_{j_{n-m}})$ form a basis of $\mathbb F^{n-m}$,
\end{enumerate}
where $\pi : \mathbb F^n \to \mathbb F^{n-m}$ is defined by
$$
\pi_1(x_1,\dots,x_n) = (x_{j_1+1},\dots,x_{j_{n-m}+1}).
$$
We proceed by induction. Suppose that for some $s \ge 1$, we have chosen $t_{j_1}, \dots, t_{j_{s-1}}$ such that:
\begin{enumerate}[label=(\alph*)]
\item the set $\{c_{i_1}, \dots, c_{i_m}, t_{j_1},\dots, t_{j_{s-1}}\}$ is linearly independent, and
\item $\pi(t_{j_1}), \dots, \pi(t_{j_{s-1}})$ are linearly independent.
\end{enumerate}
Let $U_s = \mathrm{span}\{c_{i_1}, \dots, c_{i_m}, t_{j_1},\dots,t_{j_{s-1}}\}$ and $ \dim U_s = m + (s-1)$. 
Since $s-1 < n-m$, the vectors $\pi(t_{j_1}), \dots, \pi(t_{j_{s-1}})$ do not span $\mathbb F^{n-m}$. Hence there exists $y \in \mathbb F^{n-m}$ such that $y \notin \mathrm{span}\{\pi(t_{j_1}), \dots, \pi(t_{j_{s-1}})\}$. Since $\pi$ is surjective, choose $u_s \in \mathbb F^n$ such that $\pi(u_s) = y$. Now consider the affine subspace 
$u_s + \ker(\pi) = \{u_s + w : w \in \ker(\pi)\}$. For any $w \in \ker(\pi)$, we have $\pi(u_s + w) = \pi(u_s)$, so the projection condition is preserved. We claim that there exists $w \in \ker(\pi)$ such that
$$
t_{j_s} := u_s + w \notin U_s.
$$
If not, then $u_s + \ker(\pi) \subseteq U_s$, which implies $u_s\in U_s$ and hence $\ker(\pi) \subseteq U_s$. Therefore, $y = \pi(u_s) \in \pi(U_s)$, contradicting the choice of $y$.
Thus such a $w$ exists, and we define $t_{j_s} = u_s + w$.
Proceeding inductively, we obtain vectors $t_{j_1},\dots, t_{j_{n-m}}$ such that: \begin{enumerate}
    \item the set $\{c_{i_1},c_{i_2},\dots,c_{i_m},t_{j_1},\dots,t_{j_{n-m}}\}$ forms the basis of $\mathbb F^n$, and 
    \item the set $\{\pi({t_{j_1}}),\dots,\pi(t_{j_{n-m}})\}$ forms basis of $\mathbb F^{n-m}$.
\end{enumerate}
 
The latter condition implies that the matrix $ T_\mathcal I$ is invertible. This completes the proof.
\end{proof}

\begin{proposition}\label{thm:full size nilpotent}
Let $C \in M_n(\mathbb F)$. Then the equation
$$C = x_1^k + J_{0,n} x_2^k$$
admits a solution for every integer $k \geq 2$.
\end{proposition}
\begin{proof}
Write $C = (c_1, c_2, \dots, c_n)^\top$, where $c_i$ denotes the $i$-th row of $C$. Let $T = (t_0,t_1, t_2, \dots, t_{n-1})^\top \in M_n(\mathbb{F})$ given by the columns. By the definition of $J_{0,n}$, we have
$$
C - J_{0,n}T = (c_1 - t_1,\, c_2 - t_2,\, \dots,\, c_{n-1} - t_{n-1},\, c_n)^\top.
$$
We consider two cases.\\
\textbf{Case 1 (when $c_n \neq 0$):}
Let $\operatorname{rank}(C)=m$. If $m=n$, then $C$ is invertible and hence is $k$-th power of some matrix $X_1\in M_n(\mathbb F)$. Hence, we may assume $m\leq n-1$. After reordering the rows if necessary, we may assume that $ \{c_{n-m+1}, \ldots, c_n\} $
forms a basis for the row space of $C$. Then each of $c_1, \dots, c_{n-m}$ lies in their span. Set $t_{n-m+1} = \cdots = t_{n-1} = 0$.
Then,
$$
C - J_{0,n}T = (c_1 - t_1, \dots, c_{n-m} - t_{n-m}, c_{n-m+1}, \dots, c_n)^\top.
$$
By Lemma~\ref{lem:C-LI}, we now choose $t_1, \dots, t_{n-m}$ such that the set $\{t_1, \dots, t_{n-m}, c_{n-m+1}, \dots, c_n\}$ is linearly independent. It follows that $C - J_{0,n}T$ is invertible.

Let $t_0 = e_1 \in \mathbb F^n$. By construction, the matrix $T$ satisfies the hypotheses of Lemma~\ref{lem:T-kth}. Hence, there exists $X_2 \in M_n(\mathbb{F})$ such that $T = X_2^k$. Since $C - J_{0,n}T\in \mathrm{GL}_n(\mathbb F)$, it admits a $k$-th root. Thus, there exists $X_1 \in M_n(\mathbb{F})$ such that $C - J_{0,n}T = X_1^k$. Therefore, $C = X_1^k + J_{0,n} X_2^k$. 

\textbf{Case 2 (when $c_n = 0$):}
Write $C = \begin{pmatrix} C' & * \\ 0 & 0 \end{pmatrix}$, where $C' \in M_{n-1}(\mathbb{F})$. Choose $\lambda \in \mathbb{F}$ such that $\lambda \neq 0$ and $\lambda$ is not an eigenvalue of $C'$. Define $T \in M_n(\mathbb{F})$ by specifying its rows as
$$
t_0 = e_n, \quad t_i = \lambda e_i \quad \textit{for} \quad 1\leq i\leq n-1.
$$
Then, a direct computation shows that
$$
C - J_{0,n}T = \begin{pmatrix} C' - \lambda I_{n-1} & * \\ 0 & 0
\end{pmatrix}.
$$
Since $\lambda$ is not an eigenvalue of $C'$, the matrix $C' - \lambda I_{n-1}$ is invertible. It follows that $C - J_{0,n}T$ has exactly one zero eigenvalue and all other eigenvalues are nonzero. Hence, there exists $X_1 \in M_n(\mathbb{F})$ such that $C - J_{0,n}T = X_1^k$. On the other hand, the matrix $T$ is invertible, hence admits a $k$-th root. Thus, there exists a  $X_2 \in M_n(\mathbb{F})$ such that $T = X_2^k$. Consequently, $C = X_1^k + J_{0,n} X_2^k$. 
\end{proof}

\subsection{Jordan form with one nilpotent block}
In this section, we study the equation
$$C = X_1^k + J X_2^k$$
where $J$ is in Jordan canonical form with partition $(n_1, n_2, \dots, n_r)$ satisfying $\sum_{i=1}^r n_i = n$ with $r'(J)\geq 1$ and with exactly one Jordan block corresponding to the eigenvalue $0$, i.e., $r_0(J)=1$.

\begin{lemma}\label{lem:C'inv}
For every $C' \in M_{n-n_0}(\mathbb F)$, there exists $Y \in M_{n-n_0}(\mathbb F)$ such that
$$\det\big(C' - J'Y^k\big) \neq 0.$$
\end{lemma}
\begin{proof}
Let $Y^k=\mu I_{n-n_0}$ for some scalar $\mu\in \mathbb F$. Then  $$\det(C'-J'Y^k)=\det(C'-\mu J')$$ 
is a polynomial in $\mu$ of degree at most $n- n_0$. The leading coefficient of the polynomial is equal to the product of the nonzero eigenvalues of $J'$. Hence, one can choose $\mu\in \mathbb F^\times$ such that the determinant is nonzero. Since $\mathbb F$ is algebraically closed, every scalar admits a $k$-th root, so we may write $\mu = \lambda^k$. This completes the proof.
\end{proof}
 
\begin{proposition}\label{prop:one-nil}
Let $C \in M_n(\mathbb F)$, and let $J \in M_n(\mathbb F)$ be a Jordan matrix with partition $(n_1, n_2, \dots, n_r)$ satisfying $\sum_{i=1}^r n_i = n$, with $r(J)\geq 2$ and $r_0(J)=1$. Then there exist matrices $X_1, X_2 \in M_n(\mathbb F)$ such that
$$ C = X_1^k + J X_2^k. $$
\end{proposition}
\begin{proof} Let $T' \in M_{n_0}(\mathbb F)$ and define
$$ T=\begin{pmatrix} \mu I_{n-n_0} & Q \\ 0 & T' \end{pmatrix},$$
where $\mu \in \mathbb F^\times$ and $Q \in M_{(n-n_0)\times n_0}(\mathbb F)$. Partition the matrix $C$ as 
\begin{equation*}
    C=\begin{pmatrix}
   (C_{11})_{n-n_0} & (C_{12})_{(n-n_0)\times n_0}\\
   (C_{21})_{n_0\times (n-n_0)} & (C_{22})_{n_0}
\end{pmatrix}.
\end{equation*} 
Then, 
\begin{equation*}
C-JT = \begin{pmatrix} C_{11}-\mu J' & C_{12}-J'Q  \\ C_{21} & C_{22}-J_{0,n_0}T' \end{pmatrix}. \end{equation*}
Since $J'$ is invertible, choose $Q=(J')^{-1}C_{12}$ to obtain \begin{equation*}
C-JT=\begin{pmatrix} C_{11}-\mu J' & 0  \\ C_{21} & C_{22}-J_{0,n_0}T' \end{pmatrix}.
\end{equation*} 
By Lemma~\ref{lem:C'inv}, we can choose $\mu$ such that  $C_{11}-\mu J'$ is invertible. By Lemma~\ref{lem:C-LI}, $C_{22}-J_{0,n_0}T'$ has at most one zero eigenvalue. Therefore, the matrix $C-JT$ has all eigenvalues nonzero except possibly a single zero eigenvalue, and this occurs in at most one Jordan block. Over an algebraically closed field, such a matrix admits a $k$-th root, say $X_1\in M_n(\mathbb F)$. Moreover, by Lemma~\ref{lem:T-kth}, we may choose $T' = Y^k$ for some $Y \in M_{n_0}(\mathbb F)$, and since $\mu=\lambda^k$ for some $\lambda \in \mathbb F^\times$, it follows that $T=X_2^k$ for some $X_2 \in M_n(\mathbb F)$. Therefore,
$$
C = X_1^k + JX_2^k,
$$
as required.
\end{proof} 
Using the constructions from above, we obtain the following proposition for the class of matrices having $r_0(J)\leq 1$.
\begin{proposition}
\label{prop:main_{r_0}0_1}
Let $k \geq 2$ be an integer, and $J \in M_n(\mathbb F)$ has at most one nilpotent Jordan block, i.e., $r_0(J) \le 1$. Then the map $$\widetilde{\omega}(x_1,x_2)=x_1^k+Jx_2^k$$ is surjective on $M_n(\mathbb F)$.
\end{proposition}
\begin{proof}
If $r_0(J)=0$, then $J$ is invertible, and the result follows from Proposition~\ref{thm:invertibleB}. If $r_0(J)=1$, then $J$ has exactly one nilpotent Jordan block, and the result follows from Proposition~\ref{thm:full size nilpotent} and Proposition~\ref{prop:one-nil}.
\end{proof}

\section{Jordan forms with multiple nilpotent blocks}
\label{sec: more than one nil block}
Let $J \in M_n(\mathbb{F})$ be in Jordan canonical form,
$$
J = \left(\bigoplus_{i=1}^{r'} J_{\lambda_i,n_i}\right)\oplus \left(\bigoplus_{s=1}^{r_0} J_{0,m_s}\right),
$$
where $\lambda_i\neq 0$ and $\sum_{s=1}^{r_0} m_s = n_0$. For $1\le s\le r_0$, set 
$$
\ell_s := n-n_0+m_1+\cdots+m_s,
$$
and define $\sigma:\{1,\dots,n\}\to\{1,\dots,n\}$ by
$$
\sigma(i)=
\begin{cases}
i, & 1\le i\le n-n_0,\\
i-\#\{\,s:\ell_s<i\,\}, & n-n_0<i\le n,\ i\notin\{\ell_1,\dots,\ell_{r_0}\},\\
n-r_0+s, & i=\ell_s
\end{cases}
$$
associated to the $J$ and $P$ be the permutation matrix corresponding to $\sigma$. Consider the partition of $C$ obtained under the conjugation of $P$:
$$
 P^{-1}CP =
\begin{pmatrix}
C_{11}^P & C_{12}^P\\
C_{21}^P & C_{22}^P
\end{pmatrix}.
$$
where $C_{11}^P\in M_{n-r_0}(\mathbb F)$ and $C_{22}^P\in M_{r_0}(\mathbb F)$. 
Let $W\subseteq \mathbb F^n$ be the subspace spanned by the last $r_0$ rows of $C^P$, and set $m:=\dim(W)$.

For $T=(t_1,\dots,t_n)^\top$, the $i$-th row of $JT$ is
$$
(JT)_i = \sum_{j=1}^n J_{ij} t_j,
$$
so that, for each Jordan block, $(JT)_i = \lambda t_i + t_{i+1}$ if $i$ is not the last index of the block, and $(JT)_i = \lambda t_i$ if $i$ is the last index of the block.

\subsection{Necessary condition for surjectivity}
\label{subsec: necessary condition}
We begin by recalling a result of Miller~\cite{Miller2016} on powers of nilpotent Jordan blocks.

\begin{lemma}[Miller]
\label{lem:Miller}
Let $n > k \ge 2$. Then the $k$-th power of the nilpotent Jordan block $J_{0,n}$ is conjugate to
$$
\left(\bigoplus_{k-m} J_{0,\lfloor n/k \rfloor}\right)
\oplus
\left(\bigoplus_{m} J_{0,\lceil n/k \rceil}\right),
$$
where $m \equiv n \pmod{k}$ and $0 \le m \le k-1$.
\end{lemma}

\begin{proposition}
\label{prop:nonsurjectivity_condition}
Let $k \ge 2$, and let $J \in M_n(\mathbb F)$ satisfy $r_0(J)=r_0 \ge 2$. If $n \le k(r_0 - 1)$ then the map
$\omega(x_1,x_2)=x_1^k + Jx_2^k $
is not surjective on $M_n(\mathbb F)$.
\end{proposition}
\begin{proof}
Choose $C \in M_n(\mathbb F)$ such that ${C}_{21}^P=0$ and ${C}_{22}^P = J_{0,r_0}$. Then $P^{-1}(C-JT)P$ has a Jordan block of size $r_0$ corresponding to the eigenvalue $0$. Hence its minimal polynomial is of the form $f(z)\,z^\ell$, where $f(0)\neq 0$ and $\ell \ge r_0$.

Suppose, for contradiction, that $P^{-1}(C-JT)P = X_1^k$ for some $X_1 \in M_n(\mathbb F)$. Let $p$ be the size of the Jordan block of $X_1$ corresponding to the eigenvalue $0$. By Lemma~\ref{lem:Miller}, the largest Jordan block of $X_1^k$ corresponding to $0$ has size at most $\left\lceil \frac{p}{k} \right\rceil$. Therefore,
$$
\ell \le \left\lceil \frac{p}{k} \right\rceil \le \left\lceil \frac{n}{k} \right\rceil.
$$
Combining this with $\ell \ge r_0$, we obtain $r_0 \le \left\lceil \frac{n}{k} \right\rceil$, which implies
$$
n \ge k(r_0 - 1) + 1.
$$
This contradicts the assumption $n \le k(r_0 - 1)$. Hence, no such $X_1$ exists, and $\omega$ is not surjective.
\end{proof}

\begin{remark}
Restating the above proposition, we obtain a necessary condition for the surjectivity of the map $\omega(x_1,x_2)=x_1^k+Jx_2^k$ is $r_0 < \frac{n}{k} + 1$. 
\end{remark}

\subsection{Proof of Theorem~\ref{thm:General criteria for surjectivity}}
\label{proof: thm 1.2}
For $\omega$ to be surjective, it is necessary and sufficient that $\widetilde{\omega}$ is surjective. Since $A_1$ is invertible, the matrices $A_1^{-1}A_2$ and $A_2$ have the same kernel, and hence the same nullity, by Lemma~\ref{lem:nullity-zero-blocks}. In particular, we have $r_0(A_1^{-1}A_2)=r_0(A_2)$.  Therefore, if $r_0(A_2)\leq 1$, then $r_0(A_1^{-1}A_2)\leq 1$, and the surjectivity  of $\widetilde{\omega}$ follows from Proposition~\ref{prop:main_{r_0}0_1}. This proves part~(1). Part~(2) follows directly from Proposition~\ref{prop:nonsurjectivity_condition}.
\subsection{Rank criterion for the elements in the image}\label{subsec: rank criteria}
In this subsection, we use the notation introduced at the beginning of this section.
\begin{lemma}
\label{lem: nonzero minor}
Let $C \in M_n(\mathbb F)$, and $P$ be the permutation matrix associated to the Jordan form $J$, and let 
$l_1,\dots,l_{r_0}$ be the indices corresponding to the coordinates of the nilpotent part (i.e., the Jordan blocks associated to the eigenvalue $0$) in this decomposition.
Define $ W = \mathrm{span}\{c_{l_1}, \dots, c_{l_{r_0}}\} \subset \mathbb F^n$ and $ m = \dim W$.
Assume that $m = \operatorname{rank}(C_{21}^P)$.
Then there exists $T \in M_n(\mathbb F)$ such that
$$
a_{r_0-m}
=
\sum_{\substack{S \subseteq \{1,\dots,n\} \\ |S| = n-(r_0-m)}}
\det(\widetilde{C}_S) \neq 0,
$$
where $\widetilde{C} = P^{-1}(C-JT)P$, $\widetilde{C}_S$ denotes the principal sub matrices indexed by $S$, and $a_k$ denotes the coefficient of $z^k$ in $\chi_{\widetilde C}(z)$.
\end{lemma}
\begin{proof}
The rows of $\widetilde{C}$ are given by 
$$\widetilde{c}_i=\begin{cases}
    (c_i-J_{ii}t_i-J_{i(i+1)}t_{i+1})P, &  1\leq i\leq n-n_0\\
    (c_{\sigma^{-1}(i)}-t_{\sigma^{-1}(i)})P, &  n-n_0+1\leq i\leq n-r_0\\
    c_{\sigma^{-1}(i)}P, &   n-r_0+1\leq i\leq n
\end{cases}$$
Let 
$
\mathcal I=\{S\subset\{1,\dots,n\}: |S|=n-(r_0-m)\}
$ and $\mathcal I_0=\{n-r_0+1,\dots,n\}$.
Partition $\mathcal I$ as
$$
\mathcal I_1=\{S\in \mathcal I:|S\cap \mathcal I_0|\le m\},\qquad
\mathcal I_2=\{S\in \mathcal I:|S\cap \mathcal I_0|>m\}.
$$
Since the rows of $\widetilde C$ indexed by $\mathcal I_0$ span $W$ with $\dim W=m$, any set of more than $m$ such rows is linearly dependent and hence $\det(\widetilde C_S)=0$ for all $S\in \mathcal I_2$. Thus
$$
a_{r_0-m}=\sum_{S\in \mathcal I}\det(\widetilde C_S)=\sum_{S\in \mathcal I_1}\det(\widetilde C_S).
$$
Now $|S|=n-r_0+m$ and $|\mathcal I_0|=r_0$, so for $S\in \mathcal I_1$ we have $|S\cap \mathcal I_0|\le m$, hence
$$
|S\cap\{1,\dots,n-r_0\}|\ge (n-r_0+m)-m=n-r_0,
$$
which forces $\{1,\dots,n-r_0\}\subseteq S$. Therefore each $S\in \mathcal I_1$ is of the form
$$
S=\{1,\dots,n-r_0\}\cup \mathcal J,\qquad \mathcal J\subseteq \mathcal I_0,\ |\mathcal J|=m.
$$
We now construct $T$ so that exactly one such set contributes a nonzero determinant.  

Let $\pi_1:\mathbb F^n\to\mathbb F^{n-r_0}$ and $\pi_1':\mathbb F^n\to \mathbb F^{r_0}$  given by 
\begin{align*}
   \pi_1(x_1,\dots,x_n)& =(x_1,\dots,x_{n-r_0})\\
   \pi_1'(x_1,\dots,x_n)& =(x_{n-r_0+1},\dots,x_{n})
\end{align*}
Since $\dim W=m=\operatorname{rank}(\widetilde{C}_{21}^P)$, we can choose indices $\mathcal J'=\{j_1,\dots,j_m\}\subseteq \mathcal I_0$ such that $$\{\pi_1(\widetilde c_{j_1}),\dots,\pi_1(\widetilde c_{j_m})\}$$ is linearly independent set, and set
$$
S'=\{1,\dots,n-r_0,j_1,\dots,j_m\}.
$$
First choose $\pi_1(t_{m+1}),\dots,\pi_1(t_{n-r_0})$(in reverse order) so that
$$
\{\pi_1(\widetilde c_{m+1}),\dots,\pi_1(\widetilde c_{n-r_0}),\pi_1(\widetilde c_{j_1}),\dots,\pi_1(\widetilde c_{j_m})\}
$$
forms a basis of $\mathbb F^{n-r_0}$, and $\{\pi_1(t_{m+1}),\dots,\pi_1(t_{n-r_0})\}$ is linearly independent. This can be achieved by a similar construction as in Lemma~\ref{lem:C-LI} for size $n-r_0$ matrix. Set $\pi_1'(t_i)=0$ for $m+1\leq i\leq n-r_0 $.

If $m\leq n-n_0$, since $J_{ii}\neq 0$, choose $t_i$(in reverse order) such that $J_{ii}t_iP=c_iP-\widetilde{c}_{j_i}-J_{i(i+1)}t_{i+1}P-\eta_{j_i}e_{j_i}$.
If $n-n_0<m$, then 
for $n-n_0+1\leq i\leq m$, choose $t_{\sigma^{-1}(i)}P=c_{\sigma^{-1}(i)}P-\widetilde{c}_{j_i}-\eta_{j_i}e_{j_i}$ where $\eta_{j_i}\in\mathbb F$. For $1\leq i\leq n-n_0<m$, since $J_{ii}\neq 0$, choose $t_i$(in reverse order) such that $J_{ii}t_iP=c_iP-\widetilde{c}_{j_i}-J_{i(i+1)}t_{i+1}P-\eta_{j_i}e_{j_i}$.
Then for $1\leq i\leq n-r_0$ and $i\leq m$, 
$\widetilde{c_i}=
    \widetilde{c}_{j_i}+\eta_{j_i}e_{j_i}$.

Define $\pi_2:\mathbb F^n\longrightarrow \mathbb F^{n-m}$ and $\pi_2':\mathbb F^n\longrightarrow \mathbb F^m$ given by \begin{align*}
\pi_2(x_1,x_2,\dots,x_n)&=(x_1,x_2,\dots, x_{n-r_0},x_{i_1},x_{i_2},\dots,x_{i_{r_0-m}})\\
\pi_2'(x_1,x_2,\dots,x_n)&=(x_{j_1},x_{j_2},\dots,x_{j_{m}})
\end{align*}
where $\{i_1,i_2,\dots,i_{r_0-m}\}=\mathcal I_0\setminus \mathcal J' $.

For $1\le i\le m$, 
\begin{align*}
    \pi_2(\widetilde c_i)=\pi_2(\widetilde c_{j_i}), \qquad
    \pi'_2(\widetilde c_i)=\pi'_2(\widetilde c_{j_i})+\eta_{j_i}e_{j_i}
\end{align*}
We choose $\eta_{j_i}\in\mathbb F^\times$ such that the set $\{\pi_2'(t_1),\pi_2'(t_2),\dots,\pi_2'(t_m)\}$ is a linearly independent set in $\mathbb F^m$. Since $\eta_{j_i}\neq 0$, we get $\widetilde c_i \notin \mathrm{span}\{\widetilde c_{j_1},\dots,\widetilde c_{j_m}\}$ for $1\leq i\leq m$.
We now prove that the rows indexed by $S'$ are linearly independent.
By construction,
$$
\pi_1(\widetilde c_i)=\pi_1(\widetilde c_{j_i}) \quad \text{for } 1\le i\le m,
$$
and the set
$$
\mathcal B:=\{\pi_1(\widetilde c_{m+1}),\dots,\pi_1(\widetilde c_{n-r_0}),
\pi_1(\widetilde c_{j_1}),\dots,\pi_1(\widetilde c_{j_m})\}
$$
forms a basis of $\mathbb F^{n-r_0}$.
Suppose there exists a linear relation among the rows indexed by $S'$:
$$
\sum_{i=1}^{n-r_0}\alpha_i \widetilde c_i+\sum_{\gamma=1}^m \beta_\gamma \widetilde c_{j_\gamma}=0.
$$
Applying $\pi_1$, we obtain
$$
\sum_{i=1}^{n-r_0}\alpha_i \pi_1(\widetilde c_i)+\sum_{\gamma=1}^m \beta_\gamma \pi_1(\widetilde c_{j_\gamma})=0.
$$
Splitting the sum,
\begin{align*}
&\sum_{i=1}^m \alpha_i \pi_1(\widetilde c_i)
+\sum_{i=m+1}^{n-r_0}\alpha_i \pi_1(\widetilde c_i)
+\sum_{\gamma=1}^m \beta_\gamma \pi_1(\widetilde c_{j_\gamma})=0.
\end{align*}
Using $\pi_1(\widetilde c_i)=\pi_1(\widetilde c_{j_i})$ for $1\le i\le m$, this becomes
$$
\sum_{\gamma=1}^m (\alpha_\gamma+\beta_\gamma)\pi_1(\widetilde c_{j_\gamma})
+\sum_{i=m+1}^{n-r_0}\alpha_i \pi_1(\widetilde c_i)=0.
$$
Since $\mathcal B$ is a basis of $\mathbb F^{n-r_0}$, all coefficients must vanish. Hence
$$
\alpha_i=0 \ (m+1\le i\le n-r_0), \qquad \alpha_\gamma+\beta_\gamma=0 \ (1\le \gamma\le m).
$$
Substituting back into the original relation, we get
$
\sum_{\gamma=1}^m \alpha_\gamma(\widetilde c_\gamma-\widetilde c_{j_\gamma})=0
$.
Applying $\pi_2'$, we obtain 
\begin{align*}
   \sum_{\gamma=1}^m \alpha_\gamma(\pi_2'(\widetilde c_\gamma)-\pi_2'(\widetilde c_{j_\gamma}))=0
\implies \sum_{\gamma=1}^m \alpha_{\gamma}\eta_{j_\gamma}e_{j_\gamma}=0 
\end{align*}
Since $\eta_{j_\gamma}\in\mathbb F^\times$ and $e_{j_\gamma}$ are linearly independent, $\alpha_\gamma=0$ for all $\gamma$, and therefore $\beta_\gamma=0$ as well.
Thus all coefficients vanish, and the rows indexed by $S'$ are linearly independent.
If $S=\{1,\dots,n-r_0\}\cup \mathcal J$ with $\mathcal J\neq\mathcal J'$, then $\{\widetilde c_j:j\in \mathcal J\}$ is linearly dependent since it lies in the $m$-dimensional space $W$ but is not a basis. Thus $\det(\widetilde C_S)=0$ for all such $S$.
Therefore exactly one term in the sum is nonzero, and hence $a_{r_0-m}\neq 0$.
Finally, the rows $t_1,\dots,t_{n-r_0}$ are constructed to be linearly independent. The remaining $r_0$ rows can be chosen so that they are independent modulo their span, which is always possible. Hence the rows of $T$ form a basis of $\mathbb F^n$, and $T$ is invertible.
This completes the proof.
\end{proof}
\begin{proposition}
\label{thm: rank argument}
Let $C\in M_n(\mathbb F)$, and let
$
J=\begin{pmatrix}
J' & 0 \\
0 & \oplus_{j=1}^{r_0} J_{0,m_j}
\end{pmatrix}
$
be in Jordan canonical form. Suppose that 
$
\operatorname{rank}({C}^P_{21})=\dim(W)=m.
$
Then there exist $X_1,X_2\in M_n(\mathbb F)$ such that
$$
C = X_1^k + J X_2^k.
$$
\end{proposition}
\begin{proof}
Let $T\in M_n(\mathbb F)$ be an arbitray matrix and $P$ be the permutation matrix associated to $J$.
Since $\mathrm{dim}(W)=m$, we have 
\begin{align*}
    r_0-m &\leq \mathrm{nullity}(C-JT)\leq n-m\\
m &\leq \mathrm{rank}(C-JT)\leq n-(r_0-m)
\end{align*}
The nullity equals the geometric multiplicity of the eigenvalue $0$, i.e., the number of Jordan blocks corresponding to $0$. Furthermore, the algebraic multiplicity of the eigenvalue $0$, denoted by $\mathrm{AM}_{C - JT}(0)$, satisfies:
\begin{align*}
    \mathrm{AM}_{C - JT}(0)& =n-\left(\# \hspace{0.2cm} \text{of non-zero eigenvalues counted with multiplicity}\right)\\ & \geq \operatorname{nullity}(C - JT) \geq r_0-m.
\end{align*}
The characteristic polynomial of $C-JT$ is given by $$\chi_{C-JT}(z)=z^n-a_{n-1}z^{n-1}+\dots+a_{r_0-m}z^{r_0-m},$$
where each $a_i$ is the sum of determinants of principal sub-matrix of size $(n-i).$ If we show that the coefficient $a_{r_0-m}\neq 0$, then $C-JT$ is a $k$th power matrix with $\mathrm{AM}(0)=\operatorname{nullity}(C-JT).$
By Lemma~\ref{lem: nonzero minor}, there exists $T\in M_n(\mathbb F)$ such that $a_{r_0-m}\neq 0$.

Over an algebraically closed field of characteristic $0$, any matrix whose Jordan blocks at $0$ are all of size $1$ admits a $k$-th root. Hence there exists $X_1\in M_n(\mathbb F)$ such that
$$
C-JT=X_1^k.
$$
By the construction in Lemma~\ref{lem: nonzero minor}, the rows of $T$ can be chosen to be linearly independent, and hence $T$ is invertible. Thus there exists $X_2\in M_n(\mathbb F)$ such that $T=X_2^k$.
Therefore, $
C=X_1^k+JX_2^k,
$
as required.
\end{proof}
\noindent
Proposition~\ref{thm: rank argument} shows that every matrix $C \in M_n(\mathbb F)$ satisfying
$
\operatorname{rank}(C_{21}^P) = \dim(W)
$
lies in the image of $\widetilde\omega$.

\section{Low Dimensional Cases}
In this section, we classify the surjectivity in terms of $n$, $k$, and the nullity of $A_2$ for $n=3,4$. By Lemma~\ref{lem:nullity-zero-blocks}, the nullity is equal to the number of Jordan blocks corresponding to the eigenvalue $0$. Accordingly, we denote by $r_0$ both the nullity of $A_2$ and the number of Jordan blocks of $A_2$ corresponding to the eigenvalue $0$.
\subsection{For degree \texorpdfstring{$3$}{3} matrix algebra}
\label{sec: n=3}
In this section, we specialize our study to the case of $3 \times 3$ matrices and analyze the surjectivity of the map $\omega$ in terms of the Jordan canonical form of $A_1^{-1}A_2$.
\begin{proposition}
\label{prop: n=3}
Let $A_1, A_2 \in M_3(\mathbb{F})$ and let $k \geq 2$ be an integer. Suppose that $A_1 \in \mathrm{GL}_3(\mathbb{F})$ and $r_0$ denote the number of Jordan blocks corresponding to the eigenvalue $0$ in the Jordan canonical form of $A_2$. Consider the map
$
\omega(x_1,x_2) = A_1 x_1^k + A_2 x_2^k
$ on $M_3(\mathbb{F})$.
Then $\omega$ is surjective if and only if $k(r_0 - 1) < 3$.
\end{proposition}
\begin{proof}
If $A_2 \in \mathrm{GL}_3(\mathbb{F})$, then the surjectivity of $\omega$ follows from Proposition~\ref{thm:invertibleB}. Hence, we assume that $A_2$ is not invertible.
In this case, $\omega$ is surjective if and only if the map
$$
\widetilde{\omega}(x_1,x_2) = x_1^k + (A_1^{-1}A_2)x_2^k
$$
is surjective on $M_3(\mathbb{F})$. Since the image is invariant under conjugation by $\mathrm{PGL}_3(\mathbb{F})$, we may assume that $A_1^{-1}A_2$ is in Jordan canonical form.
Thus, any non-invertible matrix is similar to one of the following:
\begin{align*}
\begin{alignedat}{3}
(i)\;& J_{0,3} \qquad 
& (ii)\;& J_{0,2} \oplus (\lambda),\ \lambda \in \mathbb{F}^\times \qquad 
& (iii)\;& J_{0,2} \oplus J_{0,1} \\
(iv)\;& (\lambda) \oplus J_{0,1} \oplus J_{0,1},\ \lambda \in \mathbb{F}^\times \qquad 
& (v)\;& (\lambda) \oplus (\mu) \oplus J_{0,1},\ \lambda,\mu \in \mathbb{F}^\times
\end{alignedat}
\end{align*}
\textbf{Case 1: $r_0=1$.}  
This corresponds to the forms $(i)$, $(ii)$, and $(v)$. In this case, the surjectivity of $\widetilde{\omega}$ follows from Proposition~\ref{prop:one-nil}.\\
\noindent\textbf{Case 2: $r_0 \geq 2$.}  
Since $k \geq 2$, the condition $k(r_0-1) < 3$ forces $r_0=2$ and $k=2$. The corresponding Jordan forms are $(iii)$ and $(iv)$.
We treat case $(iii)$ and the argument for $(iv)$ is analogous.\\
Let $C \in M_3(\mathbb{F})$. We aim to solve
$$
C = x_1^2 + 
\begin{pmatrix}
J_{0,2} & \\
& 0
\end{pmatrix}
x_2^2.
$$
Let $T = (t_0, t_1, t_2)^\top$. Then
$$
C - 
\begin{pmatrix}
J_{0,2} & \\
& 0
\end{pmatrix}
T
=
\begin{pmatrix} c_{11}-t_{11} & \begin{pmatrix} c_{12}-t_{12} & c_{13}-t_{13} \end{pmatrix}\\ (C_{21})_{2\times 1} & (C_{22})_{2\times 2} \end{pmatrix}=:C'.
$$
Let $W = \mathrm{span}\{c_2, c_3\}$.

\noindent\textbf{Subcase 2.1: $\dim(W)=2$.}  
Choose $t_1 \in \mathbb{F}^3$ such that $c_1 - t_1 \notin W$. Then the matrix above is invertible, hence equal to $X_1^2$ for some $X_1 \in M_3(\mathbb{F})$.
If $t_1 \neq 0$, extend it to a basis of $\mathbb{F}^3$ to obtain $T = X_2^2$. If $t_1 = 0$, take $T=0$.

\noindent\textbf{Subcase 2.2: $\dim(W)=1$.} 
If $\operatorname{rank}(C_{21})=1$, the result follows from Proposition~\ref{thm: rank argument}.
If $\operatorname{rank}(C_{21})=0$, we reduce $C_{22}$ to Jordan form via conjugation. If the resulting block has a nonzero eigenvalue Jordan block or is diagonal, we argue as above.
If $C_{22} \sim J_{0,2}$, then take $t_1=c_1$. By Lemma~\ref{lem:Miller}, there exists $X_1$ such that $X_1^2 = 0 \oplus J_{0,2}$. The construction of $X_2$ follows as before.

\noindent\textbf{Subcase 2.3: $\dim(W)=0$.}  
Choose $t_1 \neq 0$ so that the resulting matrix is diagonal, hence a square. Extend $t_1$ to a basis to obtain $T = X_2^2$.

This proves surjectivity when $k(r_0-1)<3$. The non-surjectivity when $k(r_0-1)\geq 3$ follows from Proposition~\ref{prop:nonsurjectivity_condition}.
\end{proof}
\begin{corollary}
\label{coro: n=3}
    Let $A$ be a non-zero element in $M_3(\mathbb F)$ and $k\geq 2$ be a positive integer. The map $\omega(x_1,x_2)=x_1^k+Ax_2^k$ is surjective on $M_3(\mathbb F)$ if and only if $k(r_0-1)<3$, where $r_0$ is the number of Jordan blocks corresponding to the eigenvalue $0$ of $A$. Moreover, when the map $\omega$ is not surjective, then image misses precisely the set $\left\{\begin{pmatrix}
        \mu & u\\
        0 & M\\
        \end{pmatrix}:\; \mu \in \mathbb F,\; u \in \mathbb F^2,\; M \in M_2(\mathbb F)\setminus\{0\},\; M^2 = 0\right\}$.
\end{corollary}
\begin{proof}
The surjectivity criterion follows immediately from Proposition~\ref{prop: n=3} by taking $A_1=I$ and $A_2=A$.

Suppose that $\omega$ is not surjective. Since $A \neq 0$, we have $r_0 \leq 2$. By Proposition~\ref{prop:main_{r_0}0_1}, the map $\omega$ is surjective when $r_0 \leq 1$, hence $r_0=2$. In this case, the condition $k(r_0-1)<3$ fails if and only if $k \geq 3$.

Thus, the non-surjective case corresponds exactly to $r_0=2$ and $k \geq 3$. For such $A$, the description of the complement of the image follows from the analysis in Case~2 of Proposition~\ref{prop: n=3}, where the only obstruction occurs when the $2\times 2$ block is nilpotent of order $2$. This yields precisely the stated set.
\end{proof}
\subsection{For degree \texorpdfstring{$4$}{4} matrix algebra}
\label{sec: n=4}
In this section, we specialize our study to the case of $4 \times 4$ matrices and analyze the surjectivity of the map $\omega$ in terms of the Jordan canonical form of $A_1^{-1}A_2$.
\begin{proposition}
\label{prop: n=4}
Let $A_1, A_2 \in M_4(\mathbb{F})$ and let $k \geq 2$ be an integer. Suppose that $A_1 \in \mathrm{GL}_4(\mathbb{F})$ and $r_0$ denote the number of Jordan blocks corresponding to the eigenvalue $0$ in the Jordan canonical form of $A_2$. Consider the map
$
\omega(x_1,x_2) = A_1 x_1^k + A_2 x_2^k
$ on $M_4(\mathbb{F})$.
Then $\omega$ is surjective if and only if $k(r_0 - 1) < 4$.
\end{proposition}
\begin{proof}
If $A_2 \in \mathrm{GL}_4(\mathbb{F})$, then the surjectivity of $\omega$ follows from Proposition~\ref{thm:invertibleB}. Hence, we assume that $A_2$ is not invertible.
In this case, $\omega$ is surjective if and only if the map
$$
\widetilde{\omega}(x_1,x_2) = x_1^k + (A_1^{-1}A_2)x_2^k
$$
is surjective on $M_4(\mathbb{F})$. Since the image is invariant under conjugation by $\mathrm{PGL}_4(\mathbb{F})$, we may assume that $A_1^{-1}A_2$ is in Jordan canonical form. Let $r_0$ denote the number of Jordan blocks corresponding to the eigenvalue $0$.
If $r_0=1$, then $A_1^{-1}A_2$ has a unique nilpotent block, and the surjectivity of $\widetilde{\omega}$ follows from Proposition~\ref{prop:one-nil}.
Assume now that $r_0 \geq 2$. We show that $\widetilde{\omega}$ is surjective whenever $k(r_0-1) < 4$. Since $r_0 \geq 2$ and $k \geq 2$, the condition $k(r_0-1) < 4$ implies that $r_0=2$ and $k \in \{2,3\}$. The possible Jordan forms of $A_1^{-1}A_2$ are:
\begin{align*}
\begin{alignedat}{2}
(i)\;& J_{0,3} \oplus J_{0,1} \qquad 
& (ii)\;& (\lambda) \oplus J_{0,2} \oplus J_{0,1},\ \lambda \in \mathbb{F}^\times\\ (iii)\;& J_{0,2} \oplus J_{0,2} & (iv)\;& (\lambda) \oplus (\mu) \oplus J_{0,1} \oplus J_{0,1},\ \lambda,\mu \in \mathbb{F}^\times 
\end{alignedat}
\end{align*}
We treat case $(i)$ and the remaining cases follow similarly.
Let $C \in M_4(\mathbb{F})$. We seek $X_1, X_2 \in M_4(\mathbb{F})$ such that
$$
C = X_1^k + (J_{0,3}\oplus J_{0,1}) X_2^k.
$$
Let $T = (t_0, t_1, t_2,t_3)^\top$. Then
$$
C - 
\begin{pmatrix}
J_{0,3} & \\
& 0
\end{pmatrix}
T
=
\begin{pmatrix} \begin{pmatrix} c_{11}-t_{11} & c_{12}-t_{12}\\c_{21}-t_{21} & c_{22}-t_{22}
\end{pmatrix} & \begin{pmatrix} c_{13}-t_{13} & c_{14}-t_{14}\\c_{23}-t_{23} & c_{24}-t_{24}
\end{pmatrix}\\ (C_{21})_{2\times 2} & (C_{22})_{2\times 2} \end{pmatrix}=:C'.
$$
Let $W = \mathrm{span}\{c_3, c_4\}$.

\noindent\textbf{Subcase 2.1: $\dim(W)=2$.}  
The proof is similar to Proposition~\ref{thm:full size nilpotent} (where $c_n\neq 0$).

\noindent\textbf{Subcase 2.2: $\dim(W)=1$.}  
If $\operatorname{rank}(C_{21})=1$, the result follows from Proposition~\ref{thm: rank argument}.
If $\operatorname{rank}(C_{21})=0$, we reduce $C_{22}$ to Jordan form via conjugation. If the resulting block has a nonzero eigenvalue Jordan block or is diagonal, we argue as above.
If $C_{22} \sim J_{0,2}$, then take $t_1=c_1$ and $t_2=c_2$. By Lemma~\ref{lem:Miller}, there exists $X_1$ such that $X_1^k = 0_2 \oplus J_{0,2}$. Let $W'=\mathrm{span}\{c_1,c_2\}$. If $\mathrm{dim}(W')=2$, then extend the basis $\{c_1,c_2\}$ of $W'$ to basis of $\mathbb F^4$ by choosing $t_0$ and $t_1$. Thus, $T$ is invertible and hence is a $k$-th power of some $X_2$. If $\mathrm{dim}(W')=1$, choose $t_0=t_3=0$. As done before, $T=X_2^k$ for some $X_2\in M_4(\mathbb F)$. 

\noindent\textbf{Subcase 2.3: $\dim(W)=0$.}  
Choose $t_0=\nu e_4$, $t_1=\nu e_1$, $t_2=\nu e_2$, $t_3=\nu e_3$ where $\nu\in \mathbb F^\times$ and is not an eigen value of $C$. Then $T$ is $k$-th power of some $X_2\in M_4(\mathbb F)$ and $C'=X_1^k$ for some $X_1\in M_4(\mathbb F)$.

Thus, $\widetilde{\omega}$ is surjective whenever $k(r_0-1) < 4$. Finally, if $k(r_0-1) \geq 4$, the map is not surjective by Proposition~\ref{prop:nonsurjectivity_condition}.
\end{proof}
\begin{corollary}
\label{coro: n=4}
Let $A$ be a nonzero element of $M_4(\mathbb F)$ and let $k \geq 2$ be an integer. Consider the polynomial map $\omega(x_1,x_2)=x_1^k+Ax_2^k$
on $M_4(\mathbb F)$. Then $\omega$ is surjective if and only if 
$k(r_0-1)<4$, where $r_0$ is the number of Jordan blocks corresponding to the eigenvalue $0$ of $A$.
\end{corollary}
\begin{proof}
The surjectivity criterion follows directly from Proposition~\ref{prop: n=4} by taking $A_1=I$ and $A_2=A$.
\end{proof}

\subsection{Proof of Theorem~\ref{thm: Low dimensional classification}}
\label{proof: thm 1.3}
Since $A_1$ is invertible, we have $r_0(A_1^{-1}A_2)=r_0(A_2)$. So $r_0$ coincides with the nullity of $A_2$. Thus, the cases $n=3$ and $n=4$ follow from Propositions~\ref{prop: n=3} and~\ref{prop: n=4}, respectively.
\subsection{Proof of Corollary~\ref{coro: map A}}
\label{proof: coro 1.4} For $A_1=I$ and $A_2=A$, $r_0(A)$ coincides with the nullity of $A$. Thus, the cases $n=3$ and $n=4$ follow from Corollary~\ref{coro: n=3} and~\ref{coro: n=4}, respectively.
\section{Conclusion}
In this paper, we studied the surjectivity of polynomial maps with non-scalar coefficients on matrix algebras, focusing on maps of the form
$$
\omega(x_1,x_2)=A_1 x_1^k + A_2 x_2^k,
$$
where $A_1$ is invertible. We established general criteria for surjectivity in terms of the number of Jordan blocks corresponding to the eigenvalue $0$ of $A_2$, and showed that these criteria are sharp in low dimensions. In particular, for $n \leq 4$, we obtained a complete characterization of surjectivity in terms of the inequality $n > k(r_0 - 1)$.

These results suggest that the interplay between polynomial maps and the nilpotent structure of the coefficient matrices provides a natural framework for the images of word maps with constants in associative algebras. The low-dimensional classification indicates that the condition obtained in \cref{thm:General criteria for surjectivity} could be optimal. The results of this paper naturally lead to the following problem.
\begin{question}
Under the assumptions of \cref{thm:General criteria for surjectivity}, the map $\omega(x_1,x_2)=A_1 x_1^k + A_2 x_2^k$ is surjective on $M_n(\mathbb F)$ if and only if
$n > k(r_0 - 1)$,
where $r_0$ denotes the number of Jordan blocks of $A_2$ corresponding to the eigenvalue $0$.
\end{question}

\printbibliography
\end{document}